\documentclass{article}
\usepackage{mathtools}
\usepackage{amsmath}
\usepackage{graphicx} 
\usepackage{amssymb}
\usepackage{dsfont}
\usepackage{amsthm}

\newtheorem{theorem}{Theorem}[section]
\newtheorem{definition}{Definition}
\newtheorem{proposition}[theorem]{Proposition}
\newtheorem{lemma}[theorem]{Lemma}
\newtheorem{remark}{Remark}[theorem]
\theoremstyle{remark}
\theoremstyle{definition}

\title{Parking on a random rooted binary tree}
\author{Semu Serunjogi}
\date{\vspace{-4ex}}
\begin{document}

\maketitle

\begin{abstract}    
    In this paper, we investigate the parking process on a uniform random rooted binary tree with $n$ vertices. Viewing each vertex as a single parking space, a random number of cars independently arrive at and attempt to park on each vertex one at a time. If a car attempts to park on an occupied vertex, it traverses the unique path on the tree towards the root, parking at the first empty vertex it encounters. If this is not possible, the car exits the tree at the root.

    We shall investigate the limit of the probability of the event that all cars can park when $\lfloor \alpha n \rfloor$ cars arrive, with $\alpha > 0$. We find that there is a phase transition at $\alpha_c = 2 - \sqrt{2}$, with this event having positive limiting probability when $\alpha < \alpha_c$, and the probability tending to 0 as $n \rightarrow \infty$ for $\alpha > \alpha_c$.

    This is analogous to the work done by Goldschmidt and Przykucki \cite{poissongw} and Goldschmidt and Chen \cite{geogw}, while agreeing with the general result proven by Curien and Hénard \cite{flux}.
\end{abstract}

\section{Introduction}

We consider a general directed planar rooted tree $t$ that has $n$ vertices. We shall write $V(t)$ for the vertices for $t$ -- note that we have a unique labeling of the vertices from left to right as $t$ is planar. We may view the vertices of $t$ as parking spaces and imagine that one-by-one $m$ cars arrive independently and uniformly randomly at and attempt to park on the vertices. We define $A: V(t) \rightarrow \{0,1,2,...\}^{n}$ as the number of cars that arrive at each vertex and consider the \emph{arrival process} $A = (A_1, A_2, \ldots, A_n)$ where $m = A_1 + \ldots + A_n$. If a car arrives at an occupied parking space, it proceeds along the unique directed path towards the root of the tree, parking at the first empty space it finds. If the car doesn't encounter an empty space, it leaves the tree.

In the case where $t$ is the path $[n] = \{1, 2, \ldots, n\}$ with edges directed towards 1, let us write $P_i$ for the preferred parking space for the $i$-th car to arrive on $[n]$. We call an arrival process $(P_1, P_2, \ldots, P_m)$ a \emph{parking function} if all of the cars can park on $[n]$.

Konheim and Weiss showed in \cite{line} that for $0 \leq m \leq n$ there are exactly $(n+1-m)(n+1)^{m-1}$ parking functions on $[n]$. If we have $m = \lfloor \alpha n \rfloor$ cars arriving where $\alpha \in (0,1)$, then the probability of $(P_1, P_2, \ldots, P_m)$ being a parking function on $[n]$ is given by

\[\frac{(n+1-m)(n+1)^{m-1}}{n^m} = \left(1-\frac{\lfloor \alpha n \rfloor - 1}{n} \right) \left(1+\frac{1}{n} \right)^{\lfloor \alpha n \rfloor} \rightarrow (1-\alpha)e^\alpha\ \]
as $n \rightarrow \infty$. Interestingly, this limiting probability is strictly positive for all $\alpha \in (0,1)$.

Returning to a general rooted tree $t$ of size $n$, we find that conditional on there being $m = \lfloor \alpha n \rfloor$ car arrivals, the arrival process $A = (A_1, A_2, \ldots, A_n)$ has a $Multinomial(m; 1/n, \ldots, 1/n)$ distribution. This converges in distribution to the Poisson distribution $Po(\alpha)$ as $n \rightarrow \infty$. This fact will later allow us to replace the tree-wide multinomial car arrival process with a local Poisson arrival process at each vertex.

The problem of parking on a range of different trees has been explored frequently in recent times. Lackner and Panholzer \cite{phase} investigated parking on a uniform random rooted labelled tree of size $n$. They showed that when $\lfloor \alpha n \rfloor$ cars arrive independently at uniformly random vertices, the limiting probability of all the cars being able to park on the tree experiences a discontinuous phase transition as we vary $\alpha$ past $1/2$. Goldschmidt and Przykucki were able to find the same results using a probabilistic method in \cite{poissongw}. Goldschmidt and Chen \cite{geogw} then considered the case of parking on a uniform random rooted plane tree, again finding a phase transition in the asymptotic parking probability as $\alpha$ varied. Jones \cite{binarygw} explored a model for the rainfall runoff from hillsides, where each tree vertex represents locations on the hill with a certain  capacity to hold water, rainwater runs down the hill according to the directed tree edges, and the probability of water reaching the bottom of the hill is considered. In \cite{flux}, Curien and Hénard were able to characterise the behaviour of this phase transition for general arrival and tree distributions. In this paper, we shall consider parking on a uniform random rooted binary tree with the same $Po(\alpha)$ car arrival distribution as in \cite{poissongw} and \cite{geogw}, and using the objective method to analyse this problem. We shall also demonstrate that our answer agrees with the general case.

We now formulate our binary tree model. Let $\mathds{V} = \bigcup_{n \geq 0} \{1,2 \}^n$ be an infinite vertex set, where $\{1,2 \}^0 = \emptyset$ by convention. We write vertices as $v = v_1 v_2 \ldots v_n$ (with $v = \emptyset$ if $n = 0$). Then $p(v) = v_1 v_2 \ldots v_{n-1}$ is the \emph{parent} of $v \neq \emptyset$, while $v1, v2$ are the \emph{children} of $v$.

\begin{definition}
    We say that a subset $t \subseteq \mathds{V}$ is a \emph{rooted binary tree} if
    \begin{enumerate}
        \item $\emptyset \in t$
        \item $v \in t \implies p(v) \in t$
        \item for every $v \in t$, there exists $0 \leq c(v) \leq 2$ such that if $c(v)= 0$ then $v1, v2 \notin t$, and if $c(v) > 0$ then $vi \in t$ if and only if $1 \leq i \leq c(v)$.
    \end{enumerate}
We take the edges of $t$ to be directed from $v$ to $p(v)$.
\end{definition}

Now we denote the set of rooted binary trees by $\mathcal{T}$. Let $\mathcal{T}_n$ be the subset of $\mathcal{T}$ containing all trees $t \in \mathcal{T}$ of size $n$. It is known that $\mathcal{T}_n$ is finite with a cardinality of $C_n = \frac{1}{n} \binom{2n}{n}$, the $n$-th Catalan number. From now on, we shall write $T$ and $T_n$ to denote uniformly random elements of $\mathcal{T}$ and $\mathcal{T}_n$, respectively.

We shall write Bin(2, 1/2) for the binomial distribution with two independent trials, each with success probability $1/2$, and Po($\lambda$) for the Poisson distribution with mean $\lambda > 0$. We write BGW(2,1/2) for the law of the family tree of a Galton-Watson branching process with Bin(2, 1/2) offspring distribution. Fix a tree $t \in \mathcal{T}_n$. Then the BGW(2,1/2) model generates $t$ with probability $\Pi_{v \in t} 2^{-c(v)-1}$. As $\Sigma_{v \in t} c(v) = n - 1$ (since $t$ has size $n$), this probability simplifies to $2^{1 - 2n}$. This probability is the same for each tree in $\mathcal{T}_n$, hence we can conclude that $T_n$ has the same distribution as a BGW(2,1/2) tree conditioned to have $n$ vertices.

In order to analyse the asymptotic probability of all the cars being able to park on $T_n$, we must consider $T_\infty$ which is $T$ conditioned to have $n$ vertices and we let $n \rightarrow \infty$. This is known as the \emph{local weak limit} of the family of trees $(T_n)^{\infty}_{n = 1}$.

\section{Relation to existing work}
First, we shall verify the following result due to Curien and Hénard in \cite{flux} for the case of the Galton-Watson (GW) trees generating uniform binary trees:
\begin{theorem}[Curien and Hénard \cite{flux}]\label{flux phi}
    Let $T$ be a GW tree with offspring distribution $\nu$. Consider the parking process of cars arriving independently at each vertex of $T$ according to the distribution $\gamma$. If $\mu \leq 1$ and $\sigma^2$ respectively are the mean and variance of $\gamma$, and if $\Sigma^2$ is the variance of the critical (mean 1) offspring distribution $\nu$, we let
    \begin{equation}\label{Phi}
        \Phi \coloneq (1-\mu)^2-\Sigma^2(\sigma^2+\mu^2-\mu).
    \end{equation}
    We then have
    \[
        \phi(T_n) \quad \begin{cases}
    \text{converges in distribution} & \Phi>0 \\ \rightarrow\infty, 
    \text{ but is } o_\mathds{P}(n) & \Phi=0 \\ \approx cn & \Phi<0
    \end{cases} \quad \text{ as } n \rightarrow \infty
    \]
    where $\phi$ is defined as the flux of the parking process on the tree, i.e. the number of cars leaving the tree through the root and $c=c(\gamma,\nu) \in (0,\mu)$ is a deterministic number.
\end{theorem}
In this paper, we have considered the parking process where $\gamma$ and $\nu$ follow the $Po(\alpha)$ and $Bin(2, 1/2)$ distributions, respectively. In this case, $\mu = \sigma^2 = \alpha$, while $\Sigma^2 = \frac{1}{2}$. Simplifying (\ref{Phi}) gives $\Phi = \frac{1}{2}(2 - 4\alpha + \alpha^2)$. If we include the restriction that $\alpha \leq 1$, then applying Theorem \ref{flux phi} allows us to conclude that the limit of the flux $\phi(T_n)$ as $n\rightarrow\infty$ undergoes a phase transition from being bounded to unbounded as $\alpha \leq 1$ increases past $2 - \sqrt{2}$. We also notice that $X$ must be \emph{stochastically increasing} in $\alpha$, hence the flux is also almost surely unbounded for $\alpha > 1$.

In this paper, our combination of offspring distribution and car arrival distribution allows us to develop a complete characterisation of the distribution of the parking process on $T_n$, which the general result of Curien and Hénard does not provide. Our method of analysis will closely follow that of \cite{poissongw} and \cite{geogw}, in which the cases of parking on uniform random rooted Cayley trees and uniform random rooted plane trees, respectively, were explored.

\subsection{Construction of $T_\infty$}
Given a parking process $A$ on a random rooted plane tree $t$, we shall assume that the car arrivals at each vertex are independent and obey the law $P \sim \gamma$, where $\gamma$ has mean $\mu$ and variance $\sigma^2 < \infty$. From now on we shall focus on critical Galton-Watson trees with each ancestor independently having $N \sim \nu$ offspring, where $\nu$ has mean 1 (hence critical) and variance $\Sigma^2 < \infty$.

By Kesten \cite{spinaldecomp}, we have the following method to construct the local weak limit tree $T_\infty$: Firstly, define the \emph{size-biased} distribution of $\nu$ shifted down by one as $\tilde{\nu}$ where
\[\tilde{\nu}(i) \coloneq (i+1)\nu(i+1) \quad \text{for} \quad i \geq 0.\]
Note that this defines a distribution of mean $\Sigma^2$ since the mean of $\nu$ is 1.

Now we let the infinite directed path $S = \{S_0, S_1, S_2, ...\}$ rooted at $S_0$ be the spine. We independently add $\tilde{N} \sim \tilde{\nu}$ offspring to each of the vertices $S_i$ on the spine. Finally, we independently attach a GW tree with offspring distribution $\nu$ to each of these neighbours of the infinite path. Kesten states that this construction does indeed give a tree with the same distribution as $T_\infty$ and that we have $T_n \stackrel{d}{\rightarrow} T_\infty$ as $n \rightarrow \infty$.

We can analyse the parking process on $T_\infty$ by first completing the parking process on the independent GW trees attached to the spine $S$, then parking any overflow and the remaining arriving cars on the spine. After parking on the trees attached to the spine, we count the number of cars that arrive at $S_h$ -- the $h$-th vertex on the spine -- coming from the GW trees with distribution $\nu$ attached to it, as well as any  cars that arrive initially at $S_h$ through the arrival process. By the construction of $T_\infty$ above, these quantities are i.i.d for each vertex $S_h$ on the spine, and are distributed according to the law of the random variable $Y$ given by
\begin{equation}\label{RDE for Y}
    Y \stackrel{d}{=} P+\sum_{i=1}^{\tilde{N}}X_i
\end{equation}
where $\tilde{N} \sim \tilde{\nu}, X_i \sim \phi(T), P \sim \gamma$ are all independent. Taking expectations in (\ref{RDE for Y}), we have $\mathds{E}[Y] = \Sigma^2 \mathds{E}[X_i] + \mu$.


\section{Main results}
We are now ready to state the main results of this paper.
\begin{theorem}\label{Thm: limit of prob A_n}
Suppose $\lfloor \alpha n \rfloor$ cars independently arrive and attempt to park at a uniform random vertex of the BGW(2,1/2) tree $T_n$ of size n.
Let $A_{n, \alpha}$ be the event that all $\lfloor \alpha n \rfloor$ of the arriving cars can park on $T_n$. Then
    \[\lim_{n \rightarrow \infty} \mathds{P}(A_{n, \alpha}) = \begin{cases}
        e^{\alpha/2} \sqrt{\frac{\alpha^2 - 4\alpha + 2}{2(1-\alpha)}} & 0 \leq \alpha \leq 2-\sqrt{2} \\
        0 & \alpha > 2-\sqrt{2}.
    \end{cases}\]
\end{theorem}
\begin{remark}
    We see from this theorem that there is a phase transition in the behaviour of this probability at the critical $\alpha$ value $\alpha_c = 2-\sqrt{2}$.
\end{remark}
\begin{theorem}\label{Thm: pgf of X}
    Let X be the number of cars that visit the root of a BGW(2,1/2) tree with, for some $\alpha > 0$, an independent $Po(\alpha)$ number of cars initially choosing each vertex. Write $p \coloneq \mathds{P}(X=0)$.
    \begin{enumerate}
        \item For $0 \leq \alpha \leq 2 - \sqrt{2}$, the probability generating function of X is
        \[G(s) = 1 - \alpha + s(2 - \alpha) + 2s^2 e^{\alpha (1-s)} + 2s e^{\alpha (1-s)} \sqrt{s^2 - e^{\alpha (s-1)} (1 - \alpha + s(2 - \alpha))}.\]
        Furthermore, $p = 1 - \alpha$ and $\mathds{E}[X] = 2 - \alpha - \sqrt{2(\alpha^2 - 4\alpha + 2)}$.
        \item For $\alpha > 2 - \sqrt{2}$, the probability generating function of X is
        \begin{align*}
            G(s) &= p + s(1-p) + 2s^2 e^{\alpha (1-s)} \\
            &+ \begin{cases} 
             2s e^{\alpha (1-s)} \sqrt{s^2 - e^{\alpha (s-1)} (p + s(1+p))}, & 0 \leq s \leq s_p \\
            - 2s e^{\alpha (1-s)} \sqrt{s^2 - e^{\alpha (s-1)} (p + s(1+p))}, & s_p < s \leq 1,
            \end{cases}
        \end{align*}
        where $s_p = \frac{1 + (1+\alpha)p - \sqrt{(1 + (1+\alpha)p)^2 - 8\alpha p(1+p)}}{2\alpha(1+p)}$. Furthermore, $p > 1-\alpha$ and $\mathds{E}[X] = \infty$.
        \item If $\alpha > 2 - \sqrt{2}$, then $\frac{e^{-\alpha}}{4} \lor (1 - \alpha) < p < c(\alpha)$ where
        \[
        c(\alpha) \coloneq 
        \begin{cases}
            \frac{(3 - 2\sqrt{2})\alpha - 1}{\alpha^2 - 6\alpha + 1} & \alpha \neq 3 + 2\sqrt{2} \\
            \frac{3\sqrt{2} - 4}{8} & \alpha = 3 + 2\sqrt{2}.
        \end{cases}
        \]
    \end{enumerate}
\end{theorem}
\begin{remark}
We observe that  $\mathds{E}[X]$ undergoes a discontinuous phase change at $\alpha = 2 - \sqrt{2}:$
    \[
    \mathds{E}[X] = \begin{cases}
        2 - \alpha - \sqrt{2(\alpha^2 - 4\alpha + 2)} & 0 \leq \alpha \leq 2 - \sqrt{2} \\
        \infty & \alpha > 2 - \sqrt{2}.
    \end{cases}
    \]
\end{remark}
\begin{theorem}\label{Thm: prob of A_alpha}
    Let $T_\infty$ be a BGW(2,1/2) tree, conditioned on non-extinction, which is constructible as in Subsection 2.1. Assume that an independent Po$(\alpha)$ number of cars arrive at each of the vertices of $T_\infty$. Let $A_\alpha$ be the event that all of the arriving cars can park on the $T_\infty$. Then
    \[\mathds{P}(A_\alpha) = \begin{cases}
        e^{\alpha/2} \sqrt{\frac{\alpha^2 - 4\alpha + 2}{2(1-\alpha)}} & 0 \leq \alpha \leq 2-\sqrt{2} \\
        0 & \alpha > 2-\sqrt{2}.
    \end{cases} \]
\end{theorem}
Before proving the first two of these theorems, we will show how Theorem \ref{Thm: limit of prob A_n} follows directly from Theorem \ref{Thm: prob of A_alpha}.

\begin{proof}[Proof of Theorem \ref{Thm: limit of prob A_n}]
    Our argument is identical to the reasoning provided in Goldschmidt and Chen \cite{geogw} to derive the Theorem 1.1 from Theorem 3.1 in that paper: It is true that conditional on the existence of a monotone coupling of the trees $T_n$ as $n$ increases, the proof of our Theorem \ref{Thm: limit of prob A_n} being a consequence of our Theorem \ref{Thm: prob of A_alpha} is identical to the proof of Theorem 1.1 in Goldschmidt and Przykucki \cite{binarygw}. The existence of such a coupling is shown by Luczack and Winkler in \cite{coupling}.
\end{proof}
Once again appealing to Kesten's construction of $T_\infty$, we shall perform the parking process on $T_\infty$ by first parking on one of the BGW(2,1/2) trees grafted to the spine $S$, and then considering the subsequent parking process on $S$. We shall let $X$ be the number of cars that arrive at the root after the parking process has been completed on this sub-tree. The iterative definition of Galton-Watson trees allows us to form a recursive distributional equation (RDE) for $X$:
\begin{equation}\label{RDE for X}
    X \stackrel{d}{=} P + \sum_{i=1}^{N}(X_i - 1)^{+},
\end{equation}
where $P \sim Po(\alpha); N \sim$ Bin(2,1/2); $X_1,X_2, \ldots$ are i.i.d copies of the random variable $X$, with all of the random variables in this equation being independent, and $x^+ \coloneq$ max($x$,0) for $x \in \mathds{R}$. Since $E[N]=1$ (so this GW tree is critical), the tree is a.s. finite, and $X$ gives an explicit construction of a solution to this equation, we have existence and uniqueness of $X$.

We shall now analyse generating functions to understand the first stage of parking on one BGW(2,1/2) tree and prove Theorem \ref{Thm: pgf of X}.
\section{Parking on a single Binomial Galton-Watson tree}
In this section, we shall first prove some results for general $\alpha > 0$. In Subsection 3.2 we will consider the $0 < \alpha \leq 2 - \sqrt{2}$ case before discussing the $\alpha > 2 - \sqrt{2}$ case in Subsection 3.3.

\subsection{Results for general $\alpha$}

We shall first establish results that hold for general $\alpha$.
\begin{proposition}\label{p>0, p=1-a or inf mean}
    Let $X$ be the number of cars arriving at the root of a BGW(2,1/2) tree and let $\alpha > 0$. Then
    \[p \coloneq \mathds{P}(X=0) \geq \frac{1}{4} e^{-\alpha} >0.\]
    Also, $\mathds{E}[X] = \infty$ or $p=1-\alpha$ (or both).
\end{proposition}

\begin{proof}
    The event of the root of the GW tree having no offspring and no cars arriving there directly necessitates $X=0$. This gives the lower bound on $p$ by independence of $N$ and $P$:
    \[p \geq \mathds{P}(N = 0, P = 0) = \frac{1}{4} \cdot \exp(-\alpha).\]

    Taking expectations in (\ref{RDE for X}) and using independence and $\mathds{E}[N]=1$ leads to
    \[\mathds{E}[X] = \alpha + \mathds{E}[X] - \mathds{E}[N] \cdot \mathds{P}(X \geq 1) =\alpha + \mathds{E}[X] - \mathds{P}(X \geq 1) \]

    If the mean of $X$ is finite, we may rearrange to get $p=1-\alpha$. Otherwise, $\mathds{E}[X] = \infty$, since $X$ is non-negative.
\end{proof}

We shall now study the probability generating function of $X$, $G(s) \coloneq \mathds{E}[s^X]$ for $s \in [0,1]$. We have
\begin{equation}\label{pgf of X eqn}
    \begin{aligned}[b]
        G(s) &= \mathds{E} \left[s^P \right] \mathds{E} \left[ \mathds{E} \left[ s^{(X-1)^+}\right]^N \right] \\
        &= \frac{1}{4} e^{\alpha(s-1)} \left(1 + \mathds{E} \left[ s^{(X-1)^+}\right]^N \right)^2 \\
        &= \frac{1}{4s^2} e^{\alpha(s-1)} \left(G(s) + s(1+p) - p \right)^2 \\
        &= \frac{1}{4s^2} a(s) \left(G(s) + b(s) \right)^2
    \end{aligned}
\end{equation}
where $a(s) = e^{\alpha(s-1)}$ and $b(s) = s(1+p) - p$. We rearrange to obtain the quadratic equation
\[
a(s)G(s)^2 + 2(a(s)b(s) - 2s^2)G(s) + a(s)b(s)^2 = 0.
\]
As $a(s) > 0$ for all $s$, solving this gives
\[
G(s) = \frac{2s^2 - a(s)b(s) \pm 2s\sqrt{s^2 - a(s)b(s)}}{a(s)}.
\]
For simplicity, we define
\begin{align}\label{Q_+}
    Q_+(s) \coloneq \frac{2s^2 - a(s)b(s) + 2s\sqrt{s^2 - a(s)b(s)}}{a(s)}
\end{align}
and
\begin{align}\label{Q_-}
    Q_-(s) \coloneq \frac{2s^2 - a(s)b(s) - 2s\sqrt{s^2 - a(s)b(s)}}{a(s)}
\end{align}
for the two possible branches of the pgf $G$. Note that as we have found an explicit solution to the RDE (\ref{RDE for X}), $G$ must exist, so we must have for all $s \in [0,1]$ \begin{align}\label{g(s) def and >= 0}
    g(s) \coloneq s^2 - a(s)b(s) = s^2 - e^{\alpha(s-1)} \left(s(1+p) - p \right) \geq 0.
\end{align}

We now give another condition on $p$.
\begin{lemma}\label{p>=1-alpha}
    For any $\alpha > 0,$ we have $p \geq 1-\alpha$.
\end{lemma}

\begin{proof}
    If $\alpha \geq 1$, then $1 - \alpha \leq 0$. The result follows immediately by applying Proposition \ref{p>0, p=1-a or inf mean}. Now we consider $0 < \alpha < 1$.
    Assume (for a contradiction) that $p < 1-\alpha$. It is clear that $g$ is infinitely differentiable with
    \[g'(s) = 2s - e^{\alpha(s-1)} \left( \alpha(s(1+p) - p) + 1 + p \right).\]
    Then $g(1) = 0$ and $g'(1) = 1 - \alpha - p > 0$, by assumption. By continuity of $g$, we must have $g(s) < 0$ for large $s < 1$. But this contradicts (\ref{g(s) def and >= 0}), thus completing the proof.
\end{proof}
We now show that $G$ is initially on the $Q_+$ branch.
\begin{lemma}\label{G = Q_+ near s=0}
    For all $\alpha > 0$ there exists some $\epsilon_\alpha > 0$ such that for $s \in [0, \epsilon_\alpha)$ we have $G(s) = Q_+(s)$.
\end{lemma}
\begin{proof}
    It is sufficient to show that $G$ cannot be given by $Q_-$ in a neighborhood around $s=0$. Notice that
    \[
    (Q_+ + Q_-)(s) = \frac{4s^2 - 2a(s)b(s)}{a(s)} = 4s^2 e^{\alpha(1-s)} - 2(s(1+p) - p)
    \]
    and so
    \begin{align}\label{(Q+ + Q-)'}
       (Q_+ + Q_-)'(s) = 4se^{\alpha(1-s)}(2 -\alpha s) -  2(1+p). 
    \end{align}
    Then $(Q_+ + Q_-)'(0) = -2(1 + p) < 0$ by Proposition \ref{p>0, p=1-a or inf mean}. This means that exactly one of $Q_+'(0)$ and $Q_-'(0)$ is negative (both cannot be negative, otherwise $Q_+$ and $Q_-$ would both be decreasing for $s > 0$ sufficiently small, but $G$ is non-decreasing). We assume for a contradiction that $Q_+'(0) < 0 < Q_-'(0)$. It can be easily verified from (\ref{Q_+}) and (\ref{Q_-}) that $Q_+(0) = Q_-(0) = p$. Then by the continuity of $Q_+$ and $Q_-$, we have $Q_+(s) < p < Q_-(s)$, for $s > 0$ sufficiently small. But as $Q_+(s) \geq Q_-(s)$ for all $s \in [0,1]$ we have a contradiction. Hence $Q_-'(s) < 0$ for $0 \leq s < \epsilon_\alpha$ for some $\epsilon_\alpha >0$, completing the proof. 
\end{proof}
We now discuss the case of small $\alpha \leq 2 - \sqrt{2}$.

\subsection{Small $\alpha$}
We now show that if $0 < \alpha \leq 2 - \sqrt{2}$, then $G(s)$ remains on the $Q_+(s)$ branch for the entire interval $[0,1]$.
\begin{lemma}\label{stay on Q+}
For $0 < \alpha \leq 2 - \sqrt{2}$, we have $G(s) = Q_+(s)$ for all $s \in [0,1]$.
\end{lemma}

\begin{proof}
    As shown in the previous lemma, the probability generating function $G(s)$ is initially on the branch $Q_+$ for $s \in [0, \epsilon_\alpha)$. As it must be continuous, $G$ can only switch branches where $g(s) = 0$, where $g$ is defined as in (\ref{g(s) def and >= 0}). We shall now show that if $0 < \alpha \leq 2 - \sqrt{2}$, $g(s) > 0$ for $s \in [0,1)$.
    \begin{equation}\label{g' <= h}
        \begin{aligned}[b]
            g'(s) &= 2s - e^{\alpha(s-1)} \left( \alpha(s(1+p) - p) + 1 + p \right) \\
            &\leq 2s - e^{\alpha(s-1)} \left( \alpha(s(2 - \alpha) + \alpha - 1) + 2 - \alpha \right) \\
            &\leq 2s - (1 + \alpha(s-1)) ( \alpha(s(2 - \alpha) + \alpha - 1) + 2 - \alpha) \\
            &= \alpha^2(2-\alpha)s^2 + (2 - \alpha(\alpha^2 -2\alpha +2) + \alpha(2 -\alpha)(\alpha-1))s \\
            &- (1-\alpha)(\alpha^2 -2\alpha +2) \\
            &=: h(s),
        \end{aligned}    
    \end{equation}
    where we used Lemma (\ref{p>=1-alpha}) (since $\alpha(s(1+p) - p) + 1 + p$ is increasing in $p$) and the inequality $e^x \geq 1 + x$ in the first and second inequalities, respectively. The function $h(s)$ is a quadratic with a root at $s=1$ and global maximum at $s = s_\text{max} \coloneq \frac{2 - \alpha(2\alpha^2 - 5\alpha + 4)}{2\alpha^2(2 - \alpha)}$. We find that $s_\text{max} \geq 1$ if and only if $\alpha^2 - 4\alpha + 2 \geq 0$, i.e. when $\alpha \leq 2 - \sqrt{2}$ or $\alpha \geq 2 + \sqrt{2}$. So, if $\alpha \leq 2 - \sqrt{2}$, then $h(s) \leq 0$ for $s \in [0,1]$ with equality only when $s=1$ and $\alpha = 2 - \sqrt{2}$. This tells us that $g'(s) < 0$ for $s \in [0,1)$. As $g(1) = 0$, we can conclude that $g(s) > 0$ for $s \in [0,1)$.
\end{proof}
We are now able to find exact forms for the quantities $p$ and $\mathds{E}[X]$ when $0 < \alpha \leq 2 - \sqrt{2}$.
\begin{lemma}\label{p = 1 - alpha, E[X] = ...}
    For all $\alpha \in (0, 2 - \sqrt{2}]$, we have $p = 1 - \alpha$ and $\mathds{E}[X] = 2 - \alpha - \sqrt{2(\alpha^2 - 4\alpha + 2)}$.
\end{lemma}
\begin{proof}
    We have shown in Lemma \ref{stay on Q+} that for this range of $\alpha$ values, the pgf $G(s) = Q_+(s)$ for all $s$. We consider the identity 
    \begin{align*}\label{Q+' identity}
        Q_+'(s) = \frac{1}{2}(Q_+ + Q_-)'(s) + \frac{1}{2}(Q_+ - Q_-)'(s).
    \end{align*}
    As seen in (\ref{(Q+ + Q-)'}), $|(Q_+ + Q_-)'(s)| < \infty$ for all $s$.
    We have
    \[
    (Q_+ + Q_-)'(s) = 4\frac{a(
s)\left(g(s)^{1/2} + \frac{1}{2}sg(s)^{-1/2}g'(s)\right) - a'(s)sg(s)^{1/2}}{a(s)^2}
    \] where we recall $a(s) = e^{\alpha(s-1)}$ and $b(s) = s(1+p)-p$ and $g$ is defined as (\ref{g(s) def and >= 0}).
    It's clear that as $s \uparrow 1$ the only potentially infinite term in this derivative is $\frac{2sg(s)^{-1/2}g'(s)}{a(s)^2}$. Assume for a contradiction that $p > 1 - \alpha$. As shown in the proof of Lemma \ref{p>=1-alpha}, $g(s) \geq 0$, $g(1) = 0$, $g'(1) = 1 - \alpha - p < 0$. Thus we can conclude that $(Q_+ + Q_-)'(1-) = -\infty$, and hence $Q_+'(1-) = -\infty$. Then Lemma \ref{stay on Q+} and Abel's Theorem gives $\mathds{E}[X] = G'(1-) = Q_+'(1-) = -\infty$. But $X \geq 0$, so $\mathds{E}[X] \geq 0$, giving a contradiction. So $p \leq 1 - \alpha$. Applying Lemma \ref{p>=1-alpha} gives $p = 1 - \alpha$.
    We now differentiate (\ref{pgf of X eqn}) and rearrange the result to give
    \begin{align}\label{G' eqn}
        G'(s) = \frac{(b(s) + G(s)) (sa'(s)(b(s) + G(s)) + 2a(s)(sb'(s) - b(s) - G(s))}{4s^3 - 2sa(s)(b(s) + G(s))}
    \end{align}
    Since the critical Galton–Watson tree BGW(2,1/2) is finite almost surely and an almost surely finite number of cars arrive at each vertex, we have $X < \infty$ a.s., so $G(1) = 1$. As $p = 1 - \alpha$, we have shown that the numerator and denominator of (\ref{G' eqn}) both tend to 0 as $s \uparrow 1$. We apply L'Hôpital's rule, take the limit as $s \uparrow 1$, and use $a'(1) = \alpha$, $a''(1) = \alpha^2$, $b(1) = 1$, $b'(s) = 1+p$ for all s, and $G(1) = 1$ to give
    \begin{align*}
        G'(1-) = \frac{(\alpha - 2)G'(1-) + 4\alpha - \alpha^2}{2 - \alpha - G'(1-)}.
    \end{align*}
    We then rearrange to obtain the quadratic equation $G'(1-)^2 + 2(\alpha - 2)G'(1-) + 4\alpha - \alpha^2 = 0$ with solutions $2 - \alpha \pm\sqrt{2(\alpha^2 - 4\alpha + 2)}$. These solutions' respective derivatives are given by
    \[
    -1 \pm \frac{\sqrt{2}\left(\alpha - 2\right)}{\sqrt{\alpha^2 - 4\alpha + 2}}.
    \]
    As $X$ is stochastically increasing in $\alpha$, we choose the root with the non-decreasing derivative and conclude that $\mathds{E}[X] = 2 - \alpha -\sqrt{2(\alpha^2 - 4\alpha + 2)}$.
\end{proof}
\subsection{Large $\alpha$}
We now consider $\alpha > 2 - \sqrt{2}$.

\begin{lemma}\label{p > 1-alpha, E[X] = infty}
If $\alpha > 2 - \sqrt{2}$, we have $p > 1-\alpha$ and $\mathds{E}[X] = \infty$.
\end{lemma}
\begin{proof}
    Clearly if $\alpha \geq 1$ then $p > 0 \geq 1-\alpha$, so consider $\alpha \in (2 - \sqrt{2}, 1)$. Suppose $p = 1 - \alpha$. Then $g'(1) = 1 - \alpha - p = 0$. Furthermore,
    \begin{align}\label{g''}
    g''(s) = 2 - \alpha e^{\alpha(s-1)} \left(\alpha(s(1+p) - p) + 2(1 + p) \right)
    \end{align}
    and
    \[
    g''(1) = 2 - \alpha(\alpha + 2 + 2p) = 2 - \alpha(\alpha + 2 + 2(1 - \alpha) = \alpha^2 - 4\alpha + 2 < 0
    \]
    as $\alpha > 2 - \sqrt{2}$. Then $g'(s) > 0$ for sufficiently large $s < 1$. As $g(1) = 0$, we must have $g(s) < 0$ in a neighbourhood of 1. This contradicts the condition (\ref{g(s) def and >= 0}). Applying Lemma \ref{p>=1-alpha} gives the first result, then Proposition \ref{p>0, p=1-a or inf mean} gives the second result.
\end{proof}
We now show that $G(s)$ must finish on the branch $Q_-$ in a neighbourhood of 1.
\begin{lemma}\label{G ends on Q-}
    If $\alpha > 2 - \sqrt{2}$, there exists some $\delta_\alpha > 0$ such that $G(s) = Q_-(s)$ for $s \in (1 - \delta_\alpha, 1]$.
\end{lemma}
\begin{proof}
    By exactly the same argument as in the proof of Lemma \ref{G = Q_+ near s=0}, we have that $|(Q_+ + Q_-)'(s)| < \infty$ and $(Q_+ - Q_-)'(1-) = -\infty$, since $p > 1 - \alpha$. We see that $Q_\pm'(1-) = \mp\infty$, so that $G = Q_-$ in a neighbourhood of 1.
\end{proof}
We now show that $G$ starts on the $Q_+$ branch, switches to the $Q_-$ branch (which it remains on) at a point in $(0,1)$.
\begin{lemma}
    Let $\alpha > 2 - \sqrt{2}$. Then there exists $t \in (0,1)$ such that
    \begin{align*}
        G(s) &= p + s(1-p) + 2s^2 e^{\alpha (1-s)} \\
        &+ \begin{cases} 
         2s e^{\alpha (1-s)} \sqrt{s^2 - e^{\alpha (s-1)} (p + s(1+p))}, & 0 \leq s \leq t \\
        - 2s e^{\alpha (1-s)} \sqrt{s^2 - e^{\alpha (s-1)} (p + s(1+p))}, & t < s \leq 1
        \end{cases}
    \end{align*}
\end{lemma}
\begin{proof}  
    Since $G(s)$ is on the $Q_+$ and $Q_-$ branches in neighbourhoods around 0 and 1, respectively, it must switch branches an odd number of times. Let the first branch switch occur at $t \in (0,1)$. By continuity of $G(s)$, these branch switches can only occur when $g(s) = 0$. As $g(s) \geq 0$, we must also have $g'(s) = 0$ at any branch switches as these are necessarily turning points. This gives the system of equations
    \[
    \begin{cases}
        s^2 - e^{\alpha(s-1)} \left(s(1+p) - p \right) = 0 \\
        2s - e^{\alpha(s-1)} \left( \alpha(s(1+p) - p) + 1 + p \right) = 0.
    \end{cases}
    \]
    We eliminate $e^{\alpha(s-1)}$ to give $\frac{s^2}{s(1+p)-p} = \frac{2s}{\alpha(s(1+p)-p) + 1 + p}$ which we rearrange into the quadratic equation $f(s) \coloneq \alpha(1+p)s^2 - (1+(1+\alpha)p)s +2p = 0$. This has at most two roots in $(0,1)$, so $G$ can only switch branches at most twice. We know that $t$ must be a root of $f$. As $G$ must have an odd number of branch switches, $G$ switches branches exactly once, giving the stated form for $G$.
\end{proof}
We continue with our analysis of $g(s) = s^2 - e^{\alpha(s-1)} \left(s(1+p) - p \right)$.
\begin{lemma}\label{g has two turning points}
    Let $\alpha > 2 - \sqrt{2}$. Then $g$ has two turning points, $t_1$ and $t_2$, which satisfy $0 < t_1 < t_2 < 1$. Both $t_1$ and 1 are roots of $g$, and $t_1$ is the point at which $G$ switches branches. Furthermore, $g$ is monotonically decreasing on $(0, t_1)$, increasing on $(t_1, t_2)$, and decreasing on $(t_2, 1)$.
\end{lemma}
\begin{proof}
    We have
    \begin{align*}
       g'''(s) &= - \alpha^2 e^{\alpha(s-1)} \left(\alpha(s(1+p) - p) + 3(1 + p) \right) \leq 0
    \end{align*}
for all $s \in (0,1)$, so $g''$ is monotonically decreasing on $(0,1)$. We recall that since $\alpha > 2 - \sqrt{2}$, we have $p > 1 - \alpha$ by Lemma \ref{p > 1-alpha, E[X] = infty}. Using (\ref{g''}) gives
\begin{align*}
    g''(0) &= 2 - \alpha e^{-\alpha} (2 + (2 - \alpha)p) \geq 2 - 4\alpha e^{-\alpha} \geq 2 - \frac{4}{e} > 0 \\
    g''(1) &= 2 - \alpha(\alpha + 2 + 2p) < 2 - \alpha(\alpha + 2 + 2(1 - \alpha) = \alpha^2 - 4\alpha + 2 < 0,
\end{align*}
as $\alpha > 2 - \sqrt{2}$. Hence, $g''$ has one root in $(0,1)$. This means that $g'$ is monotonically increasing, then decreasing in $(0,1)$, so $g$ has at most two turning points in $(0,1)$. We have shown that $g$ is strictly positive on $(0,1)$, except for either one or two points where $g$ and $g'$ both equal zero. Let $t_1$ be the leftmost of these. Then $g(0) = p e^{-\alpha} > 0$, $g(t_1)=0$ and $g(1)=0$. To ensure that $g(s) \geq 0$ on $(t_1,1)$, there must be a turning point $t_2 \in (t_1, 1)$ where $g(t_2) > 0$. As $g$ cannot have three turning points in $(0,1)$, there are no more roots of $g$ in $(t_1, 1)$. The description of $g$ given in the statement of the lemma follows.
\end{proof}
We can now identify the value of the branch switching point $s_p$ in Theorem \ref{Thm: pgf of X}.
\begin{lemma}\label{p bound and s_p}
Let $\alpha > 2 - \sqrt{2}$. Then the point $t$ where G(s) switches branch is the smaller root of the quadratic function $f$.
\end{lemma}
\begin{proof}
    The result is immediate if the root is repeated. Suppose $f$ has two distinct roots, and let $r$ be the other root. We have shown that $r \in (0,1)$, so $g(r) \geq 0$. Moreover, by Lemma \ref{g has two turning points}, $t$ is the only root of $g$ in $(0,1)$. Hence $g(r) > 0$.
    It can be shown that the smallest root of $f$ must be greater than $\frac{p}{1+p}$, so $r > \frac{p}{1+p}$ also. Then
    \begin{align*}
        g'(r) &= 2r - e^{\alpha(r-1)} \left(\alpha(r(1+p) + 1 + p \right) \\ &> 2r - \frac{r^2}{r(1+p)-p} \left(\alpha(r(1+p) + 1 + p \right) \\
        &= -\frac{r}{r(1+p)-p} \left(\alpha(1+p)r^2 - (1+(1+\alpha)p)r +2p \right) \\
        &= -\frac{r}{r(1+p)-p} f(r) \\ &= 0
    \end{align*}
    where the inequality comes from combining $g(r) > 0$ and $r > \frac{p}{1 + p}$ to give $e^{\alpha(r-1)} < \frac{r^2}{r(1+p)-p}$, and the final equality is due to $r$ being a root of $f$. By Lemma \ref{g has two turning points}, we can conclude that $r$ must be greater than $t = t_1$.
\end{proof}
\begin{remark}
    This allows us to obtain the form of $t = s_p$ given in Theorem \ref{Thm: pgf of X} by calculating the smaller root of the quadratic $f(s) = 0$.
\end{remark}
We now find bounds on the value of $p$.
\begin{lemma}
If $\alpha > 2 - \sqrt{2}$, then $\frac{e^{-\alpha}}{4} \lor (1 - \alpha) < p < c(\alpha)$ where
\[
c(\alpha) \coloneq 
\begin{cases}
    \frac{(3 - 2\sqrt{2})\alpha - 1}{\alpha^2 - 6\alpha + 1} & \alpha \neq 3 + 2\sqrt{2} \\
    \frac{3\sqrt{2} - 4}{8} & \alpha = 3 + 2\sqrt{2}.
\end{cases}
\]
\end{lemma}
\begin{proof}
   The lower bound on $p$ follows from Proposition \ref{p>0, p=1-a or inf mean} and Lemma \ref{p > 1-alpha, E[X] = infty}.
   We have shown that the point $t \in (0,1)$ at which $G$ switches branch is a root of $f$, so its discriminant $(\alpha^2 - 6\alpha + 1)p^2 + 2(1-3\alpha)p + 1$ must be non-negative. Since $f(0) = 2p$ and $f(1) = \alpha + p -1$ are both strictly positive by Proposition \ref{p>0, p=1-a or inf mean} and Lemma \ref{p > 1-alpha, E[X] = infty}, both (including repeated) roots of $f$ lie in $(0,1)$. Suppose $\alpha \in (2 - \sqrt{2}, 3 + 2\sqrt{2})$, then $\alpha^2 - 6\alpha + 1 < 0$. This means that the discriminant of $f$ is a quadratic function of $p$ with a global maximum and roots given by $p_\pm = \frac{(3 \pm 2\sqrt{2})\alpha - 1}{\alpha^2 - 6\alpha + 1}$ where $p_+ < 0 < p_-$. The requirement that the discriminant and $p$ are both non-negative leads to the upper bound $p \leq p_-$. It is easy to verify that $c(\alpha)$ as defined in the statement of the lemma is the continuation of $p_-$ (as a function of $\alpha$) to the domain $(2 - \sqrt{2}, \infty)$ and when $\alpha = 3 + 2\sqrt{2}$ we have that $p_+$ is undefined so $p \leq c(3 + 2\sqrt{2})$. If $\alpha > 3 + 2\sqrt{2}$, then $f$ has a global minimum and $p_+ > p_- > 0$, so either $p \leq p_- = c(\alpha)$ or $p \geq p_+$. For small $\alpha > 3 + 2\sqrt{2}$, we have $p_+ > 1$, so we must have $p \leq c(\alpha)$ for such $\alpha$. As $\alpha$ increases, $p_\pm \rightarrow 0$ but $p_+ > p_-$. Hence it is impossible for $p$ to make a discontinuous jump from $p \leq p_-$ to $p \geq p_+$, as this would require $p$ to increase alongside $\alpha$ at some point. Hence, $p \leq c(\alpha)$ for all $\alpha > 2 - \sqrt{2}$. 
\end{proof}

\section{Parking on $T_\infty$}
We are now able to study the car parking process on all of $T_\infty$. We shall use the notation introduced in subsection 2.1 where for the spine $S = \{S_0, S_1, \ldots\}$  of $T_\infty$ that is rooted at $S_0$, $Y_h$ is the number of cars that arrive at the $h$-th vertex $S_h$ on $S$ coming either from any of the GW trees attached to $S_h$, or from any cars that initially attempt to park at $S_h$.

Then by equation (\ref{RDE for Y}) with $\gamma \sim Po(\alpha)$ we have that $Y_1, Y_2, \ldots$ have a common distribution satisfying the distributional equation
\begin{equation}\label{RDE for Y2}
    Y \stackrel{d}{=} P+\sum_{i=1}^{\tilde{N}}X_i
\end{equation}
where $P \sim Po(\alpha)$, $\mathds{P}(\tilde{N} = 0) = \mathds{P}(\tilde{N} = 1) = \frac{1}{2}$, and $X_1, X_2, \ldots$ are i.i.d copies of $X$, with all of these random variables being independent.

A key realisation in the parking process is that all of the cars can park on $S$ if and only if
\[C_n \coloneq n - \sum_{i=1}^{n}Y_i \geq 0 \text{ for all } n \in \mathds{N}.\]

This is due to the first $n$ vertices of $S$ representing $n$ parking spaces, and there are at least $\sum_{i=1}^{n}Y_i$ cars attempting to park in these spaces. So, if $C_n$ is negative for some $n$ then $(Y_i)_{i\geq 0}$ does not define a parking function on $S$. Conversely, if we don't have a parking function for $S$, then there is a minimal $n$ such that the cars starting their journey on $S$ cannot all park on that initial segment of the path, resulting in $C_n<0$. This tells us that $\mathds{P}(A_\alpha) = \mathds{P}(C_n \geq 0 \text{ for all } n \in \mathds{N}).$

We are almost ready to prove Theorem \ref{Thm: prob of A_alpha}. First, we state a lemma from \cite{skip} that will be useful in our proof:
\begin{lemma}\label{RW}
    Let $(R_n)_{n \geq 1}$ be a random walk with i.i.d step sizes $W_1, W_2, \ldots$ such that $\mathds{P}(W_1 \leq 1)=1, \mathds{E}[W_1]=m \geq 0, \mathds{P}(W_1 = 1)=q>0$. Then
    \[\mathds{P}(R_n \geq 0 \text{ for all } n \in \mathds{N}) = \frac{m}{q}.\]
\end{lemma}
We now prove Theorem \ref{Thm: prob of A_alpha}.
\begin{proof}[Proof of Theorem \ref{Thm: prob of A_alpha}]
    Taking expectations in (\ref{RDE for Y2}),
    \begin{align*}
        \mathds{E}[Y] &= \mathds{E}[P] + \mathds{E}\left[\Tilde{N}\right]\cdot\mathds{E} \left[(X-1)^+\right] \\
        &= \alpha + \frac{1}{2} (\mathds{E}[X]-(1-p)) \\
        &= \begin{cases}
        1 - \frac{1}{2}\sqrt{ 2(\alpha^2 - 4\alpha +2) } & \text{ if } 0 \leq \alpha \leq 2 - \sqrt{2} \\
        \infty & \text{ if } \alpha > 2 - \sqrt{2}
    \end{cases}
    \end{align*}
    since $\mathds{E} \left[(X-1)^+\right] = \mathds{E}[X]-(1-p)$ and $p=1-\alpha$ when $\alpha \leq 2 - \sqrt{2}$, and using Lemma (\ref{p > 1-alpha, E[X] = infty}) for $\alpha > 2 - \sqrt{2}$. Independence of the random variables in the RDE (\ref{RDE for X}) gives
    \[
    \mathds{P}(X=0) = \mathds{P}(P=0)\cdot\mathds{P}\left(\sum_{i=1}^{N} (X_i-1)^+=0\right)
    \]
    so that
    \begin{align}\label{P(X=0)}
        \mathds{P}\left(\sum_{i=1}^{N} (X_i-1)^+=0\right) = pe^\alpha.
    \end{align}
    Similarly, (\ref{RDE for Y2}) leads to
    \begin{equation}
        \begin{aligned}[b]\label{P(Y=0)}
        \mathds{P}(Y=0) &= \mathds{P}(P=0)\cdot\mathds{P}\left(\sum_{i=1}^{\Tilde{N}} (X_i-1)^+=0\right) \\
        &= e^{-\alpha}\cdot\mathds{P}\left(\sum_{i=1}^{\Tilde{N}} (X_i-1)^+=0\right)
        \end{aligned}    
    \end{equation}
    We may condition on the value $\Tilde{N}$ takes, as follows:
    \begin{equation}\label{Sum1}
        \begin{aligned}[b]            \mathds{P}\left(\sum_{i=1}^{\Tilde{N}} (X_i-1)^+=0\right) &= \sum_{n=0}^1 \mathds{P}\left(\sum_{i=1}^{\Tilde{N}} (X_i-1)^+=0 \mid \Tilde{N}=n\right)\cdot\mathds{P}(\Tilde{N}=n) \\
        &= \sum_{n=0}^1 \mathds{P}\left( (X-1)^+=0\right)^n \cdot\mathds{P}(\Tilde{N}=n) \\
        &= G_{\tilde{N}}(\rho) \\
        &= \frac{1}{2} \left(1 + \rho \right),
        \end{aligned}
    \end{equation}
    where $\rho = \mathds{P}\left( (X-1)^+=0\right)$, $G_{\tilde{N}}(\rho)$ is the pgf of $\tilde{N}$ evaluated at $\rho$, and by using the fact that $\sum_{i=1}^{n} (X_i-1)^+=0$ if and only if $(X_i-1)^+=0$ for all $i = 1, \ldots, n$, along with the independence of the $X_i$. Similarly, we can condition on the value of $N$:
    \begin{equation}\label{Sum2}
        \begin{aligned}[b]
            \mathds{P}\left(\sum_{i=1}^{N} (X_i-1)^+=0\right) &= G_{N}(\rho) = \frac{1}{4} \left(1 + \rho \right)^2
        \end{aligned}
    \end{equation} 
    where $G_N(\rho)$ is the pgf of $N$ evaluated at $\rho$. Combining (\ref{P(X=0)}) and (\ref{Sum2}) gives $\rho = -1 + 2p^{\frac{1}{2}}e^{\frac{\alpha}{2}}$. Substituting this into (\ref{Sum1}) then (\ref{P(Y=0)}) leads to $\mathds{P}(Y=0) = p^{\frac{1}{2}}e^{-\frac{\alpha}{2}}$.
    If $\alpha \leq 2 - \sqrt{2}$, then $p=1-\alpha$ and we may now apply Lemma (\ref{RW}) to give
    \[\mathds{P}(A_\alpha) = \mathds{P}(C_n \geq 0 \text{ for all } n \in \mathds{N}) = \frac{\mathds{E}[1-Y_n]}{\mathds{P}(1-Y_n=1)} = e^{\frac{\alpha}{2}}\sqrt{\frac{\alpha^2 - 4\alpha + 2}{2(1-\alpha)}}.\]

    Finally, for $\alpha > 2 - \sqrt{2}$, by the stochastic monotonicity in $\alpha$, we obtain
    \[0 \leq \mathds{P}(A_\alpha) = \mathds{P}(C_n \geq 0 \text{ for all } n \in \mathds{N}) \leq \inf_{\alpha \in (0, 2 - \sqrt{2}]} e^{\frac{\alpha}{2}}\sqrt{\frac{\alpha^2 - 4\alpha + 2}{2(1-\alpha)}} = 0,\]
    giving $\mathds{P}(A_\alpha) = 0$.
\end{proof}
\begin{remark}
    We see that $\mathds{P}(A_\alpha)$ is continuous for all $\alpha\in[0,1]$ and that $\mathds{P}(A_0)=1$, as expected.
\end{remark}

\section{Acknowledgements}
I would like to thank Christina Goldschmidt for introducing me to parking functions and supporting me through my undergraduate dissertation that inspired this work. I am also grateful to them for providing useful feedback of a first draft of this paper.

\bibliography{main}

\bibliographystyle{ieeetr}

\end{document}